\documentclass{aomart}
\pdfoutput=1

\theoremstyle{plain}
\newtheorem{theorem}{Theorem}

\newtheorem{lemma}[theorem]{Lemma}

\numberwithin{theorem}{section}

\numberwithin{equation}{section}

\newcommand{\R}{{\mathbb R}}

\newcommand{\Z}{{\mathbb Z}}

\title[The sphere packing problem in dimension 24]{%
  \texorpdfstring
    {The sphere packing problem\\ in dimension $24$}
    {The sphere packing problem in dimension 24}%
}

\author[Cohn]{Henry Cohn}
\address{Microsoft Research New England\\
Cambridge, MA} \email{cohn@microsoft.com}

\author[Kumar]{Abhinav Kumar}
\address{Stony Brook University\\
Stony Brook, NY}
\email{thenav@gmail.com}

\author[Miller]{Stephen D.\ Miller}
\address{Rutgers University\\
Piscataway, NJ} \email{miller@math.rutgers.edu}

\author[Radchenko]{Danylo Radchenko}
\address{Max Planck Institute for Mathematics\\
Bonn, Germany}
\curraddr{The Abdus Salam International Centre\hfill\break\indent for Theoretical Physics,
Trieste, Italy}
\email{danradchenko@gmail.com}

\author[Viazovska]{Maryna Viazovska}
\address{Berlin Mathematical School \textup{and}
Humboldt University of Berlin,\hfill\break\indent
Berlin, Germany}
\curraddr{\'Ecole Polytechnique F\'ed\'erale de Lausanne,\hfill\break\indent Lausanne, Switzerland}
\email{viazovska@gmail.com}

\thanks{Miller's research was supported by National Science Foundation grants
DMS-1500562 and CNS-1526333.}

\received{\formatdate{2016-5-23}}

\copyrightnote{\textcopyright~2017 by the authors.
This paper may be reproduced, in its entirety, for noncommercial purposes.}

\volumenumber{185}
\issuenumber{3}
\publicationyear{2017}
\papernumber{8}
\startpage{1017}
\endpage{1033}

\begin{document}

\begin{abstract}
Building on Viazovska's recent solution of the sphere packing problem in
eight dimensions, we prove that the Leech lattice is the densest packing of
congruent spheres in twenty-four dimensions and that it is the unique
optimal periodic packing.  In particular, we find an optimal auxiliary
function for the linear programming bounds, which is an analogue of
Viazovska's function for the eight-dimensional case.
\end{abstract}

\maketitle

\section{Introduction}

The sphere packing problem asks how to arrange congruent balls as
\mbox{densely} as possible without overlap between their interiors.  The
\emph{density} is the fraction of space covered by the balls, and the problem
is to find the maximal possible density.  This problem plays an important
role in geometry, number theory, and information theory. See \cite{SPLAG} for
background and references on sphere packing and its applications.

Although many interesting constructions are known, provable optimality is
very rare.  Aside from the trivial case of one dimension, the optimal density
was previously known only in two \cite{T}, three \cite{H}, \cite{FPK}, and
eight \cite{V} dimensions, with the latter result being a recent breakthrough
due to Viazovska; see \cite{C}, \cite{dLV} for expositions. Building on her
work, we solve the sphere packing problem in twenty-four dimensions:

\begin{theorem}\label{thm:main}
The Leech lattice achieves the optimal sphere packing density in $\R^{24}$,
and it is the only periodic packing in $\R^{24}$ with that density, up to
scaling and isometries.
\end{theorem}

In particular, the optimal sphere packing density in $\R^{24}$ is that of
the Leech lattice, namely
\[
\frac{\pi^{12}}{12!} = 0.0019295743\dots\,.
\]
For an appealing construction of the Leech lattice, see Section~2.8 of
\cite{E}.

It is unknown in general whether optimal packings have any special structure,
but our theorem shows that they do in $\R^{24}$. The optimality and
uniqueness of the Leech lattice were previously known only among lattice
packings \cite{CK}, which is a far more restrictive setting. Recall that a
lattice is a discrete subgroup of $\R^n$ of rank $n$, and a lattice packing
uses spheres centered at the points of a lattice, while a periodic packing is
the union of finitely many translates of a lattice. Lattices are far more
algebraically constrained, and it is widely believed that they do not achieve
the optimal density in most dimen\-sions. (For example, see
\cite[p.~140]{SPLAG} for an example in $\R^{10}$ of a periodic packing that
is denser than any known lattice.) By contrast, periodic packings at least
come arbitrarily close to the optimal sphere packing density.

The proof of \fullref{Theorem}{thm:main} will be based on the linear
programming bounds for sphere packing, as given by the following theorem.

\begin{theorem}[Cohn and Elkies \cite{CE}] \label{thm:LP}
Let $f \colon \R^n \to \R^n$ be a Schwartz function and $r$ a positive real
number such that $f(0) = \widehat{f}(0)=1$, $f(x) \le 0$ for $|x| \ge r$, and
$\widehat{f}(y) \ge 0$ for all $y$.  Then the sphere packing density in
$\R^n$ is at most
\[
\frac{\pi^{n/2}}{(n/2)!}\left(\frac{r}{2}\right)^n.
\]
\end{theorem}

Here $(n/2)!$ means $\Gamma(n/2+1)$ when $n$ is odd, and the Fourier
transform is normalized by
\[
\widehat{f}(y) = \int_{\R^n} f(x) e^{-2\pi i \langle x,y \rangle} \, dx,
\]
where $\langle \cdot, \cdot \rangle$ denotes the usual inner product on
$\R^n$.  Without loss of generality, we can radially symmetrize $f$, in
which case $\widehat{f}$ is radial as well. We will often tacitly identify
radial functions on $\R^{24}$ with functions on $[0,\infty)$ and vice versa,
by using $f(r)$ with $r \in [0,\infty)$ to denote the common value $f(x)$
with $|x|=r$. All Fourier transforms will be in $\R^{24}$ unless otherwise
specified. In other words, if $f$ is a function of one variable defined on
$[0,\infty)$, then $\widehat{f}(r)$ means
\[
\int_{\R^{24}} f\big(|x|\big) e^{-2\pi i \langle x,y \rangle} \, dx,
\]
where $y \in \R^{24}$ satisfies $|y|=r$.

Optimizing the bound from \fullref{Theorem}{thm:main} requires choosing the
right aux\-iliary function $f$.  It was not previously known how to do so
except in one dimension \cite{CE} or eight \cite{V}, but Cohn and Elkies
conjectured the existence of an auxiliary function proving the optimality of
the Leech lattice \cite{CE}. We prove this conjecture by developing an
analogue for the Leech lattice of Viazovska's construction for the $E_8$ root
lattice.

In the case of the Leech lattice, proving optimality amounts to achieving
$r=2$, which requires that $f$ and $\widehat{f}$ have roots on the spheres of
radius $\sqrt{2k}$ about the origin for $k=2,3,\dots$\,.  See \cite{CE} for
further explanation and discussion of this condition. Furthermore, the
argument in Section~8 of \cite{CE} shows that if $f$ has no other roots at
distance $2$ or more, then the Leech lattice is the unique optimal periodic
packing in $\R^{24}$.  Thus, the proof of \fullref{Theorem}{thm:main} reduces
to constructing such a function.

The existence of an optimal auxiliary function in $\R^{24}$ has long been
anticipated, and Cohn and Miller made further conjectures in \cite{CM} about
special values of the function, which we also prove.  Our approach is based
on a new connection with quasimodular forms discovered by Viazovska \cite{V},
and our proof techniques are analogous to hers.  In \fullref{Sections}{sec:a}
and~\ref{sec:b} we will build two radial Fourier eigenfunctions in $\R^{24}$,
one with eigenvalue $1$ constructed using a weakly holomorphic quasimodular
form of weight $-8$ and depth $2$ for $\mathrm{SL}_2(\Z)$, and one with
eigenvalue $-1$ constructed using a weakly holomorphic modular form of weight
$-10$ for the congruence subgroup $\Gamma(2)$.  We will then take a linear
combination of these eigenfunctions in \fullref{Section}{sec:proof} to
construct the optimal auxiliary function. Throughout the paper, we will make
free use of the standard definitions and notation for modular forms from
\cite{V}, \cite{Z}.


\vspace{-0.1cm}

\section{\texorpdfstring{The $+1$ eigenfunction}{The +1 eigenfunction}}
\label{sec:a}

\vspace{-0.05cm}

We begin by constructing a radial eigenfunction of the Fourier transform in
$\R^{24}$ with eigenvalue $1$ in terms of the quasimodular form
\begin{equation} \label{eq:phidef}
\begin{split}
\varphi &= \frac{\big(25E_4^4 - 49E_6^2E_4\big) + 48E_6E_4^2E_2 + \big({-49}E_4^3 + 25E_6^2\big)E_2^2}{\Delta^2}\\
&= -3657830400 q
 - 314573414400 q^2
 - 13716864000000 q^3+ O\big(q^4\big),
\end{split}
\end{equation}
where $q = e^{2\pi i z}$ and the variable $z$ lies in the upper half plane.
As mentioned in the introduction, we follow the notation of \cite{V}. In
particular, $E_k$ denotes the Eisenstein series
\[
E_k(z) = 1 + \frac{2}{\zeta(1-k)}\sum_{n=1}^\infty \sum_{d \,\mid\, n} d^{k-1} e^{2\pi i n z},
\]
which is a modular form of weight $k$ for $\mathrm{SL}_2(\Z)$ when $k$ is
even and greater than $2$ (and a quasimodular form when $k=2$). Furthermore,
we normalize $\Delta$ by
\[
\Delta = \frac{E_4^3-E_6^2}{1728} = q - 24q^2 + 252q^3 + O\big(q^4\big).
\]
Recall that $\Delta$ vanishes nowhere in the upper half plane.

This function $\varphi$ is a weakly holomorphic quasimodular form of weight
$-8$ and depth $2$ for the full modular group.  Specifically, because
\[
z^{-2}E_2\mathopen{}\left(-\frac{1}{z}\right)\mathclose{} = E_2(z) - \frac{6i}{\pi z},
\]
we have the quasimodularity relation
\begin{equation}
\label{eq:quasimodular}
z^8 \varphi\mathopen{}\left(-\frac{1}{z}\right)\mathclose{} = \varphi(z) + \frac{\varphi_1(z)}{z} + \frac{\varphi_2(z)}{z^2},
\end{equation}
where
\begin{align*}
\varphi_1 &= -\frac{6i}{\pi} 48\frac{E_6E_4^2}{\Delta^2} - \frac{12i}{\pi} \frac{E_2 \big({-49E_4^3} + 25E_6^2\big)}{\Delta^2}\\
&= \frac{i}{\pi}\Big(725760 q^{-1}
 + 113218560
 + 19691320320 q+ O\big(q^2\big)\Big)
\end{align*}
and
\begin{align*}
\varphi_2 &= - \frac{36\big({-49E_4^3} + 25E_6^2\big)}{\pi^2\Delta^2}\\
&= \frac{1}{\pi^2}\Big(864 q^{-2}
 + 2218752 q^{-1}
 + 223140096
 + 23368117248 q+ O\big(q^2\big)\Big).
\end{align*}

It follows from setting $z=it$ in \eqref{eq:quasimodular} that
\begin{equation} \label{eq:phiovertinfty}
\varphi(i/t) = O\big(t^{-10} e^{4\pi t}\big)
\end{equation}
as $t \to \infty$, while the $q$-series \eqref{eq:phidef} for $\varphi$ shows
that
\begin{equation} \label{eq:phiovertzero}
\varphi(i/t) = O\big(e^{-2\pi/t}\big)
\end{equation}
as $t \to 0$.  We define
\begin{equation} \label{eq:adef}
a(r) = - 4 \sin\mathopen{}\big(\pi r^2/2\big)^2\mathclose{} \int_0^{i \infty} \varphi\mathopen{}\left(-\frac{1}{z}\right)\mathclose{} z^{10} e^{\pi i r^2 z} \, dz
\end{equation}
for $r>2$, which converges absolutely by these bounds.

\begin{lemma} \label{lemma:plusone}
The function $r \mapsto a(r)$ analytically continues to a holomorphic
function on a neighborhood of $\R$.  Its restriction to $\R$ is a Schwartz
function and a radial eigenfunction of the Fourier transform in $\R^{24}$
with eigenvalue $1$.
\end{lemma}

\begin{proof}
We follow the approach of \cite{V}, adapted to use modular forms of
different weight. Substituting
\[
-4\sin\mathopen{}\big(\pi r^2/2\big)^2\mathclose{} = e^{\pi i r^2} - 2 + e^{-\pi i r^2}
\]
yields
\begin{align*}
a(r) &= \int_{-1}^{i \infty -1} \varphi\mathopen{}\left(-\frac{1}{z+1}\right)\mathclose{} (z+1)^{10} e^{\pi i r^2 z} \, dz
- 2 \int_0^{i \infty} \varphi\mathopen{}\left(-\frac{1}{z}\right)\mathclose{} z^{10} e^{\pi i r^2 z} \, dz\\
& \quad \phantom{} + \int_1^{i \infty+1} \varphi\mathopen{}\left(-\frac{1}{z-1}\right)\mathclose{} (z-1)^{10} e^{\pi i r^2 z} \, dz\\
&=\int_{-1}^{i} \varphi\mathopen{}\left(-\frac{1}{z+1}\right)\mathclose{} (z+1)^{10} e^{\pi i r^2 z} \, dz\\
&\quad \phantom{} + \int_{i}^{i \infty} \varphi\mathopen{}\left(-\frac{1}{z+1}\right)\mathclose{} (z+1)^{10} e^{\pi i r^2 z} \, dz\\
& \quad \phantom{} - 2 \int_0^{i} \varphi\mathopen{}\left(-\frac{1}{z}\right)\mathclose{} z^{10} e^{\pi i r^2 z} \, dz - 2 \int_i^{i \infty} \varphi\mathopen{}\left(-\frac{1}{z}\right)\mathclose{} z^{10} e^{\pi i r^2 z} \, dz\\
& \quad \phantom{} + \int_1^{i} \varphi\mathopen{}\left(-\frac{1}{z-1}\right)\mathclose{} (z-1)^{10} e^{\pi i r^2 z} \, dz\\
&\quad \phantom{}  + \int_i^{i \infty} \varphi\mathopen{}\left(-\frac{1}{z-1}\right)\mathclose{} (z-1)^{10} e^{\pi i r^2 z} \, dz,
\end{align*}
where we have shifted contours as in the proof of Proposition~2 in \cite{V}.
Now the quasimodularity relation \eqref{eq:quasimodular} and periodicity
modulo $1$ show that
\begin{align*}
&\varphi\mathopen{}\left(-\frac{1}{z+1}\right)\mathclose{} (z+1)^{10} - 2  \varphi\mathopen{}\left(-\frac{1}{z}\right)\mathclose{} z^{10} + \varphi\mathopen{}\left(-\frac{1}{z-1}\right)\mathclose{} (z-1)^{10}
\\
& \qquad = \varphi(z+1) (z+1)^2 - 2 \varphi(z) z^2 + \varphi(z-1) (z-1)^2\\
& \qquad\qquad \phantom{} + \varphi_1(z+1) (z+1) - 2 \varphi_1(z) z + \varphi_1(z-1) (z-1)\\
& \qquad\qquad \phantom{} + \varphi_2(z+1) - 2 \varphi_2(z) + \varphi_2(z-1)\\
& \qquad = 2 \varphi(z).
\end{align*}
Thus,
\begin{equation}
\label{eq:conta}
\begin{split}
a(r) & =
\int_{-1}^{i} \varphi\mathopen{}\left(-\frac{1}{z+1}\right)\mathclose{} (z+1)^{10} e^{\pi i r^2 z} \, dz\\
& \quad \phantom{} + \int_1^{i} \varphi\mathopen{}\left(-\frac{1}{z-1}\right)\mathclose{} (z-1)^{10} e^{\pi i r^2 z} \, dz\\
& \quad \phantom{} - 2 \int_0^{i} \varphi\mathopen{}\left(-\frac{1}{z}\right)\mathclose{} z^{10} e^{\pi i r^2 z} \, dz + 2 \int_i^{i \infty} \varphi(z) e^{\pi i r^2 z} \, dz,
\end{split}
\end{equation}
which gives the analytic continuation of $a$ to a neighborhood of $\R$ by
\eqref{eq:phiovertinfty} and \eqref{eq:phiovertzero}. Essentially the same
estimates as in Proposition~1 of \cite{V} show that it is a Schwartz
function.  Specifically, the exponential decay of $\varphi(z)$ as the
imaginary part of $z$ tends to infinity suffices to bound all the terms in
\eqref{eq:conta}, which shows that $a$ and all its derivatives decay
exponentially.

Taking the $24$-dimensional radial Fourier transform commutes with the
integrals in \eqref{eq:conta} and amounts to replacing $e^{\pi i r^2 z}$ with
$z^{-12} e^{\pi i r^2 (-1/z)}$.  Therefore
\begin{align*}
\widehat{a}(r) & =
\int_{-1}^{i} \varphi\mathopen{}\left(-\frac{1}{z+1}\right)\mathclose{} (z+1)^{10} z^{-12} e^{\pi i r^2 (-1/z)} \, dz\\
& \quad \phantom{} + \int_1^{i} \varphi\mathopen{}\left(-\frac{1}{z-1}\right)\mathclose{} (z-1)^{10} z^{-12} e^{\pi i r^2 (-1/z)} \, dz\\
& \quad \phantom{} - 2 \int_0^{i} \varphi\mathopen{}\left(-\frac{1}{z}\right)\mathclose{} z^{-2} e^{\pi i r^2 (-1/z)} \, dz + 2 \int_i^{i \infty} \varphi(z) z^{-12} e^{\pi i r^2 (-1/z)} \, dz.
\end{align*}
Now setting $w = -1/z$ shows that
\begin{align*}
\widehat{a}(r) & =
\int_{1}^{i} \varphi\mathopen{}\left(-1-\frac{1}{w-1}\right)\mathclose{} \left(-\frac{1}{w}+1\right)^{10} w^{10} e^{\pi i r^2 w} \, dw\\
& \quad \phantom{} + \int_{-1}^{i} \varphi\mathopen{}\left(1-\frac{1}{w+1}\right)\mathclose{} \left(-\frac{1}{w}-1\right)^{10} w^{10} e^{\pi i r^2 w} \, dw\\
& \quad \phantom{} + 2 \int_i^{i\infty} \varphi(w) e^{\pi i r^2 w} \, dw - 2 \int_0^{i} \varphi\mathopen{}\left(-\frac{1}{w}\right)\mathclose{} w^{10} e^{\pi i r^2 w} \, dw.
\end{align*}
Thus, \eqref{eq:conta} and the fact that $\varphi$ is periodic modulo $1$
show that $\widehat{a} = a$, as desired.
\end{proof}

For $r>2$, we have
\begin{equation} \label{eq:a2}
a(r) = 4 i \sin\mathopen{}\big(\pi r^2/2\big)^2\mathclose{} \int_0^\infty \varphi(i/t) t^{10} e^{-\pi r^2 t} \, dt
\end{equation}
by \eqref{eq:adef}. By the quasimodularity relation \eqref{eq:quasimodular},
\begin{equation} \label{eq:10overt}
t^{10} \varphi(i/t) = t^2 \varphi(it) - i t \varphi_1(it) - \varphi_2(it).
\end{equation}
Thanks to the $q$-expansions with $q = e^{- 2 \pi t}$, we have
\begin{equation} \label{eq:phip}
t^{10} \varphi(i/t) = p(t)
+ O\big(t^2 e^{-2\pi t}\big)
\end{equation}
as $t \to \infty$, where
\[
p(t) =
-\frac{864}{\pi^2} e^{4\pi t}
+ \frac{725760}{\pi} t e^{2\pi t}
-\frac{2218752}{\pi^2} e^{2\pi t}
+ \frac{113218560}{\pi} t
- \frac{223140096}{\pi^2}.
\]
Let
\begin{align*}
\widetilde p(r) &= \int_0^\infty p(t) e^{-\pi r^2 t} \, dt\\
& = -\frac{864}{\pi^3(r^2-4)}+\frac{725760}{\pi^3(r^2-2)^2}-\frac{2218752}{\pi^3(r^2-2)}+
\frac{113218560}{\pi^3r^4}-\frac{223140096}{\pi^3r^2}.
\end{align*}
Then
\begin{equation} \label{eq:avalues}
a(r) = 4 i \sin\mathopen{}\big(\pi r^2/2\big)^2\mathclose{} \left( \widetilde p(r) + \int_0^\infty \big(\varphi(i/t) t^{10} - p(t)\big) e^{-\pi r^2 t} \, dt\right)
\end{equation}
for $r>2$. The integral
\[
\int_0^\infty \big(\varphi(i/t) t^{10} - p(t)\big) e^{-\pi r^2 t} \, dt
\]
is analytic on a neighborhood of $[0,\infty)$, and hence \eqref{eq:avalues}
holds for all $r$.  Note in particular that $a$ maps $\R$ to $i\R$ by
\eqref{eq:avalues} (or by \eqref{eq:adef} via analytic continuation).

\eqfullref{Equation}{eq:avalues} implies that $a(r)$ vanishes to second order
whenever $r = \sqrt{2k}$ with $k > 2$, because $\widetilde p$ has no poles at
these points. Furthermore, this formula implies that
\begin{flalign*}
&& a(0) &= \frac{113218560i}{\pi}, &&\\
&& a\big(\sqrt{2}\big) &= \frac{725760i}{\pi}, &&\\
&& a'\big(\sqrt{2}\big) &= \frac{-4437504\sqrt{2}i}{\pi}, &&\\
&& a(2) &= 0, &&\\
&\text{and}\hidewidth\\
&& a'(2) &= \frac{-3456i}{\pi}. &&
\end{flalign*}
The Taylor series expansion is
\[
a(r) = \frac{113218560i}{\pi} - \frac{223140096i}{\pi} r^2 + O\big(r^4\big)
\]
around $r=0$.

If we rescale $a$ so that its value at $0$ is $1$, then the value at
$\sqrt{2}$ becomes $1/156$ and the derivative there becomes
$-107\sqrt{2}/2730$, and the derivative at $2$ becomes $-1/32760$.  The
Taylor series becomes
\[
1 -\frac{3587}{1820} r^2 + O\big(r^4\big).
\]
However, the higher order terms in this Taylor series do not appear to be
rational, because they involve contributions from the integral in
\eqref{eq:avalues}.

We arrived at the definition \eqref{eq:phidef} of $\varphi$ via the Ansatz
that $\Delta^2 \varphi$ should\break be a holomorphic quasimodular form of
weight $16$ and depth $2$ for $\mathrm{SL}_2(\Z)$.\break  The space of such
forms is five-dimensional, spanned by $E_4^4$, $E_6^2E_4$, $E_6E_4^2E_2$,
$E_4^3E_2^2$, and $E_6^2E_2^2$. Within this space, one can solve for
$\varphi$ in several ways.  We initially found it by matching the numerical
conjectures from \cite{CM}, but in retrospect one can instead impose
constraints on its behavior at $0$ and $i \infty$, namely,
\eqref{eq:phiovertinfty} and \eqref{eq:phiovertzero}.  This information is
enough to determine $\varphi$ and hence the eigenfunction $a$, up to a
constant factor.

\section{\texorpdfstring{The $-1$ eigenfunction}{The -1 eigenfunction}}
\label{sec:b}

Next we construct a radial eigenfunction of the Fourier transform in
$\R^{24}$ with eigenvalue $-1$.  We will use the notation
\begin{flalign*}
&& \Theta_{00}(z) & = \sum_{n \in \Z} e^{\pi i n^2 z}, &&\\
&& \Theta_{01}(z) & = \sum_{n \in \Z} (-1)^n e^{\pi i n^2 z}, &&\\
&\text{and}\hidewidth\\
&& \Theta_{10}(z) & = \sum_{n \in \Z} e^{\pi i (n+1/2)^2 z} &&
\end{flalign*}
for theta functions from \cite{V}.  These functions satisfy the
transformation laws
\begin{alignat*}{5}
\Theta_{00}^4 |_{2} S &= -\Theta_{00}^4, &\qquad \Theta_{01}^4 |_{2} S &= -\Theta_{10}^4, &\qquad \Theta_{10}^4 |_{2} S &= -\Theta_{01}^4,\\
\Theta_{00}^4 |_{2} T &= \Theta_{01}^4, &\qquad \Theta_{01}^4 |_{2} T &= \Theta_{00}^4, &\qquad \Theta_{10}^4 |_{2} T &= -\Theta_{10}^4,
\end{alignat*}
where $S = \begin{pmatrix}0 & -1\\ 1 & 0
\end{pmatrix}$, $T =
\begin{pmatrix}1 & 1\\ 0 & 1\end{pmatrix}$, and
\[
\big(g |_k M\big)(z) = (cz+d)^{-k} g\mathopen{}\left(\frac{az+b}{cz+d}\right)\mathclose{}
\]
for a function $g$ on the upper half plane and a matrix $M = \begin{pmatrix}a & b\\
c & d
\end{pmatrix} \in \mathrm{SL}_2(\R)$.

Let
\begin{equation} \label{eq:psiIdef}
\begin{split}
\psi_I &= \frac{7 \Theta_{01}^{20} \Theta_{10}^8 + 7 \Theta_{01}^{24} \Theta_{10}^4 + 2 \Theta_{01}^{28}}{\Delta^2}\\
&= 2q^{-2} - 464q^{-1} + 172128 - 3670016q^{1/2} + 47238464q\\
& \quad \phantom{}  - 459276288q^{3/2}  + O\big(q^2\big),
\end{split}
\end{equation}
which is a weakly holomorphic modular form of weight $-10$ for $\Gamma(2)$,
and let
\begin{equation} \label{eq:psiSdef}
\begin{split}
\psi_S &= \psi_I |_{-10} S = -\frac{7 \Theta_{10}^{20} \Theta_{01}^8 + 7 \Theta_{10}^{24} \Theta_{01}^4 + 2 \Theta_{10}^{28}}{\Delta^2}\\
& = -7340032q^{1/2} - 918552576q^{3/2} + O\big(q^{5/2}\big)
\end{split}
\end{equation}
and
\begin{align*}
\psi_T &= \psi_I |_{-10} T = \frac{7 \Theta_{00}^{20} \Theta_{10}^8 - 7 \Theta_{00}^{24} \Theta_{10}^4 + 2 \Theta_{00}^{28}}{\Delta^2}\\
& = 2q^{-2} - 464q^{-1} + 172128 + 3670016q^{1/2}\\
& \quad \phantom{} + 47238464q + 459276288q^{3/2} + O\big(q^2\big).
\end{align*}
Note that $\psi_S + \psi_T = \psi_I$, which follows from the Jacobi identity
${\Theta_{01}^4 + \Theta_{10}^4 = \Theta_{00}^4}$.

Using these $q$-expansions, we find that
\begin{equation} \label{eq:psitinfinity}
\psi_I(it) = O\big(e^{4\pi t}\big)
\end{equation}
as $t \to \infty$, and
\begin{equation} \label{eq:psitzero}
\psi_I(it) = O\big(t^{10} e^{-\pi/t}\big)
\end{equation}
as $t \to 0$.  Let
\[
b(r) = -4 \sin\mathopen{}\big(\pi r^2/2\big)^2\mathclose{} \int_0^{i \infty} \psi_I(z) e^{\pi i r^2 z} \, dz
\]
for $r>2$, where the integral converges by the above bounds.

\begin{lemma} \label{lemma:minusone}
The function $r \mapsto b(r)$ analytically continues to a holomorphic
function on a neighborhood of $\R$.  Its restriction to $\R$ is a Schwartz
function and a radial eigenfunction of the Fourier transform in $\R^{24}$
with eigenvalue $-1$.
\end{lemma}

\begin{proof}
As in the proof of Proposition~6 from \cite{V}, we substitute
\[
-4\sin\mathopen{}\big(\pi r^2/2\big)^2\mathclose{} = e^{-\pi i r^2} - 2 +
e^{\pi i r^2}
\]
and shift contours to show that for $r>2$,
\begin{align*}
b(r) & = \int_{-1}^{i\infty-1} \psi_I(z+1) e^{\pi i r^2 z} \, dz - 2 \int_0^{i \infty} \psi_I(z) e^{\pi i r^2 z} \, dz\\
& \quad \phantom{} + \int_1^{i \infty+1} \psi_I(z-1) e^{\pi i r^2 z} \, dz\\
&=
\int_{-1}^i \psi_T(z) e^{\pi i r^2 z} \, dz + \int_{1}^i \psi_T(z) e^{\pi i r^2 z} \, dz - 2\int_{0}^i \psi_I(z) e^{\pi i r^2 z} \, dz\\
& \quad \phantom{} + 2 \int_i^{i\infty} \big(\psi_T(z)-\psi_I(z)\big) e^{\pi i r^2 z} \, dz.
\end{align*}
Here, we have used $\psi_I(z+1) = \psi_I(z-1) = \psi_T(z)$, and we have
shifted the endpoints from $i \infty \pm 1$ to $i \infty$ (which is
justified because the inequality $r>2$ ensures that the integrand decays
exponentially).  Finally, applying $\psi_T - \psi_I = -\psi_S$ yields
\begin{align*}
b(r) & =
\int_{-1}^i \psi_T(z) e^{\pi i r^2 z} \, dz + \int_{1}^i \psi_T(z) e^{\pi i r^2 z} \, dz - 2\int_{0}^i \psi_I(z) e^{\pi i r^2 z} \, dz\\
& \quad \phantom{} - 2 \int_i^{i\infty} \psi_S(z) e^{\pi i r^2 z} \, dz,
\end{align*}
which yields the analytic continuation to $r \le 2$, and essentially the
same estimates prove that it is a Schwartz function.

To show that the $24$-dimensional radial Fourier transform $\widehat{b}$
satisfies ${\widehat{b}=-b}$, we follow the approach of Proposition~5 from
\cite{V}.  As in the proof of \fullref{Lemma}{lemma:plusone},
\begin{align*}
\widehat{b}(r) & =
\int_{-1}^i \psi_T(z) z^{-12} e^{\pi i r^2 (-1/z)} \, dz + \int_{1}^i \psi_T(z) z^{-12} e^{\pi i r^2 (-1/z)} \, dz\\
& \quad \phantom{} - 2\int_{0}^i \psi_I(z) z^{-12} e^{\pi i r^2 (-1/z)} \, dz - 2 \int_i^{i\infty} \psi_S(z) z^{-12} e^{\pi i r^2 (-1/z)} \, dz,
\end{align*}
and the change of variables $w=-1/z$ yields
\begin{align*}
\widehat{b}(r) & =
\int_{1}^i \psi_T\mathopen{}\left(-\frac{1}{w}\right)\mathclose{} w^{10} e^{\pi i r^2 w} \, dw + \int_{-1}^i \psi_T\mathopen{}\left(-\frac{1}{w}\right)\mathclose{} w^{10} e^{\pi i r^2 w} \, dw\\
& \quad \phantom{} + 2\int_{i}^{i \infty} \psi_I\mathopen{}\left(-\frac{1}{w}\right)\mathclose{} w^{10} e^{\pi i r^2 w} \, dw + 2 \int_0^{i} \psi_S\mathopen{}\left(-\frac{1}{w}\right)\mathclose{} w^{10} e^{\pi i r^2 w} \, dw.
\end{align*}
Finally, $\widehat{b}=-b$ follows from the equations
\[
\psi_I|_{-10} S = \psi_S, \quad \psi_S |_{-10} S = \psi_I, \quad \text{and}\quad \psi_T|_{-10} S
= - \psi_T,
\]
where the first two equations amount to the definition of $\psi_S$ and the
third follows from $\psi_S + \psi_T = \psi_I$.
\end{proof}

For $r>2$, we have
\begin{equation} \label{eq:b2}
b(r) = -4i \sin\mathopen{}\big(\pi r^2/2\big)^2\mathclose{} \int_0^\infty \psi_I(it) e^{-\pi r^2 t}\, dt.
\end{equation}
From the $q$-expansion, we have
\[
\psi_I(it) = 2e^{4\pi t} - 464e^{2\pi t} + 172128 + O\big(e^{-\pi t}\big)
\]
as $t \to \infty$, and
\[
\int_0^\infty \left(2e^{4\pi t} - 464e^{2\pi t} + 172128\right) e^{-\pi r^2 t}\, dt =
\frac{2}{\pi(r^2-4)}-\frac{464}{\pi(r^2-2)}+\frac{172128}{\pi r^2}.
\]
Thus, for all $r \ge 0$,
\begin{align*}
b(r) &= -4 i \sin\mathopen{}\big(\pi r^2/2\big)^2\mathclose{} \bigg(\frac{2}{\pi(r^2-4)}-\frac{464}{\pi(r^2-2)}+\frac{172128}{\pi r^2}\\
&\quad \phantom{} + \int_0^\infty \left(\psi_I(it) - 2e^{4\pi t} + 464e^{2\pi t} - 172128\right) e^{-\pi r^2 t}\, dt \bigg),
\end{align*}
by analytic continuation.

\enlargethispage{\baselineskip}

This formula implies that $b(r)$ vanishes to second order whenever $r =
\sqrt{2k}$ with $k > 2$. Furthermore, it implies that
\begin{flalign*}
&& b(0) &= b\big(\sqrt{2}\big) = b(2) = 0, &&\\
&& b'\big(\sqrt{2}\big) &= 928 i \pi \sqrt{2}, &&\\
&\text{and}\hidewidth\\
&& b'(2) &= -8 \pi i. &&
\end{flalign*}
The Taylor series expansion is
\[
b(r) = -172128\pi i r^2 + O\big(r^4\big)
\]
around $r=0$, and $b$ maps $\R$ to $i\R$.

To obtain the definition \eqref{eq:psiIdef} of $\psi_I$, we began with the
Ansatz that $\Delta^2\psi_I$ should be a holomorphic modular form of weight
$14$ for $\Gamma(2)$.  The space of such forms is eight-dimensional, spanned
by $\Theta_{01}^{4i}\Theta_{10}^{28-4i}$ with $i = 0, 1, \dots, 7$, and the
subspace of forms satisfying the linear constraint $\psi_S + \psi_T = \psi_I$
is three-dimensional.  As in the case of $\varphi$ in
\fullref{Section}{sec:a}, one can solve for $\psi_I$ in several ways. In
particular, within the subspace satisfying $\psi_S + \psi_T = \psi_I$, the
asymptotic behavior specified by \eqref{eq:psitinfinity} and
\eqref{eq:psitzero} determines $\psi_I$ up to a constant factor.


\vspace{-0.2cm}

\section{\texorpdfstring{Proof of \fullref{Theorem}{thm:main}}{Proof of Theorem 1.1}}
\label{sec:proof}

\vspace{-0.15cm}

We can now construct the optimal auxiliary function for use in
\fullref{Theorem}{thm:LP}. Let
\[
f(r) = -\frac{\pi i}{113218560} a(r) - \frac{i}{262080\pi}b(r).
\]
Then $f(0) = \widehat{f}(0) = 1$, and the quadratic Taylor coefficients of
$f$ and $\widehat{f}$ are $-14347/5460$ and $-205/156$, respectively, as
conjectured in \cite{CM}. The functions $f$ and $\widehat{f}$ have roots at
all of the vector lengths in the Leech lattice, i.e., $\sqrt{2k}$ for
$k=2,3,\dots$\,. These roots are double roots except for the root of $f$ at
$2$, where $f'(2) = -1/16380$ (in accordance with Lemma~5.1 in \cite{CM}).
Furthermore, $f$ has the value $1/156$ and derivative $-146\sqrt{2}/4095$ at
$\sqrt{2}$, while $\widehat{f}$ has the value $1/156$ and derivative
$-5\sqrt{2}/117$ there.

We must still check that $f$ satisfies the hypotheses of
\fullref{Theorem}{thm:LP}.  We will do so using the approach of \cite{V},
with one extra complication at the end.

\enlargethispage{\baselineskip}

For $r>2$, \eqfullref{equations}{eq:a2} and \eqref{eq:b2} imply that
\[
f(r) = \sin\mathopen{}\big(\pi r^2/2\big)^2\mathclose{} \int_0^\infty A(t) e^{-\pi r^2 t} \, dt,
\]
where
\begin{align*}
A(t) &= \frac{\pi}{28304640} t^{10} \varphi(i/t) - \frac{1}{65520\pi} \psi_I(it)\\
&= \frac{\pi}{28304640} t^{10} \varphi(i/t) + \frac{1}{65520\pi} t^{10}\psi_S(i/t).
\end{align*}
To show that $f(r) \le 0$ for $r\ge 2$ with equality only at $r$ of the form
$\sqrt{2k}$ with $k = 2, 3, \dots$, it suffices to show that $A(t) \le 0$.
Specifically, $A$ cannot be identically zero since then $f$ would vanish as
well; given that $A$ is continuous, nonpositive everywhere, and negative
somewhere, it follows that
\[
\int_0^\infty A(t) e^{-\pi r^2 t} \, dt < 0
\]
for all $r$ for which it converges (i.e., $r>2$).

Because
\[
A(t) = \frac{\pi}{28304640} t^{10} \left(\varphi(i/t) + \frac{432}{\pi^2} \psi_S(i/t)\right),
\]
showing that $A(t) \le 0$ amounts to showing that
\begin{equation} \label{eq:Aineq}
\varphi(it) + \frac{432}{\pi^2} \psi_S(it) \le 0.
\end{equation}
The formula
\[
\psi_S = -\frac{7 \Theta_{10}^{20} \Theta_{01}^8 + 7
\Theta_{10}^{24} \Theta_{01}^4 + 2 \Theta_{10}^{28}}{\Delta^2}
\]
immediately implies that $\psi_S(it) \le 0$, and so to prove \eqref{eq:Aineq}
it suffices to prove that $\varphi(it) \le 0$.  We prove this inequality in
\fullref{Lemma}{lemma:phinonpos} by bounding the truncation error in the
$q$-series and examining the leading terms (splitting into the cases $t \ge
1$ and $t \le 1$). It follows that $f(r) \le 0$ for $r \ge 2$, as desired.

For $r>2$, the analogous formula for $\widehat{f}$ is
\begin{equation} \label{eq:fhatB}
\widehat{f}(r) = \sin\mathopen{}\big(\pi r^2/2\big)^2\mathclose{} \int_0^\infty B(t) e^{-\pi r^2 t} \, dt,
\end{equation}
where
\begin{equation} \label{eq:Bdef}
\begin{split}
B(t) &= \frac{\pi}{28304640} t^{10} \varphi(i/t) + \frac{1}{65520\pi} \psi_I(it)\\
&= \frac{\pi}{28304640} t^{10} \varphi(i/t) - \frac{1}{65520\pi} t^{10}\psi_S(i/t).
\end{split}
\end{equation}
To show that $\widehat{f}(r) \ge 0$ for $r>2$, it suffices to show that $B(t)
\ge 0$ for all $t \ge 0$, i.e.,
\begin{equation} \label{eq:Bineq}
\varphi(it) - \frac{432}{\pi^2} \psi_S(it) \ge 0,
\end{equation}
for the same reason as we saw above with $A(t)$. This inequality is
\fullref{Lemma}{lemma:ineq2}.

The formula \eqref{eq:fhatB} in fact holds for $r > \sqrt{2}$, not just
$r>2$.  To see why, we must examine the asymptotics of $B(t)$.  There is no
problem with the integral in \eqref{eq:fhatB} as $t \to 0$, because $B(t)$
vanishes in this limit by \eqref{eq:phiovertzero} and \eqref{eq:psitzero}.
However, the exponential growth of $B(t)$ as $t \to \infty$ causes divergence
when $r$ is too small for $e^{-\pi r^2 t}$ to counteract this growth.  To
estimate the growth rate, note that by \eqref{eq:phip} and
\eqref{eq:psiIdef}, the $e^{4\pi t}$ terms cancel in the asymptotic expansion
of $B(t)$ as $t \to \infty$, which means that $B(t) = O\big(t e^{2\pi
t}\big)$. Thus, the formula \eqref{eq:fhatB} for $\widehat{f}(r)$ converges
when $r > \sqrt{2}$, and it must equal $\widehat{f}(r)$ by analytic
continuation. Note that it cannot hold for the whole interval $(0,\infty)$,
because that would force $\widehat{f}$ to vanish at $\sqrt{2}$, which does
not happen.

Thus, \eqref{eq:fhatB} and the inequality $B(t) \ge 0$ in fact prove that
$\widehat{f} \ge 0$ for all $r \ge \sqrt{2}$. When $0 < r < \sqrt{2}$, this
inequality no longer implies that $\widehat{f}(r) \ge 0$, which is a
complication that does not occur in \cite{V}. Instead, we must analyze $B(t)$
more carefully. As $t \to \infty$, \eqfullref{equations}{eq:phip} and
\eqref{eq:psiIdef} show that
\[
B(t) = \frac{1}{39}te^{2\pi t}-\frac{10}{117\pi}e^{2\pi t} + O(t).
\]
We will ameliorate this behavior by subtracting these terms over the
interval $[1,\infty)$. They contribute
\[
\int_1^\infty \left(\frac{1}{39}te^{2\pi t}-\frac{10}{117\pi}e^{2\pi t}\right) e^{-\pi r^2 t} \, dt =
\frac{(10-3\pi)(2-r^2) + 3}{117\pi^2(r^2-2)^2}e^{-\pi(r^2-2)},
\]
which is nonnegative for $0 < r < \sqrt{2}$, and the remaining terms
\[
\int_0^1 B(t) e^{-\pi r^2 t} \, dt + \int_1^\infty \left(B(t) - \frac{1}{39}te^{2\pi t}+\frac{10}{117\pi}e^{2\pi t} \right) e^{-\pi r^2 t} \, dt
\]
converge for all $r>0$.  The integrand $B(t)$ in the first integral is
nonnegative, and thus to prove that $\widehat{f}(r) \ge 0$ for $0 < r <
\sqrt{2}$ it suffices to prove that
\begin{equation} \label{eq:Basymp}
B(t) \ge \frac{1}{39}te^{2\pi t}-\frac{10}{117\pi}e^{2\pi t}
\end{equation}
for $t \ge 1$, which is \fullref{Lemma}{lemma:Bleadingterms}.

Combining the results of this section shows that $f$ satisfies the hypotheses
of \fullref{Theorem}{thm:LP}, and thus that the Leech lattice is an optimal
sphere packing in $\R^{24}$.  Furthermore, $f$ has no roots $r>2$ other than
$r=\sqrt{2k}$ with $k=2,3,\dots$, and as in Section~8 of \cite{CE} this
condition implies that the Leech lattice is the unique densest periodic
packing in $\R^{24}$.  This completes the proof of
\fullref{Theorem}{thm:main}.

\appendix
\section{Inequalities for quasimodular forms}
\label{app:ineq}

The proof in \fullref{Section}{sec:proof} requires checking certain
inequalities for quasimod\-ular forms on the imaginary axis. Fortunately,
these inequalities are not too delicate, because equality is never attained.
The behavior at infinity is easily analyzed, which reduces the proof to
verifying the inequalities on a compact interval, and that can be done by a
finite calculation.

Thus, these inequalities are clearly provable if true. The proof of the
analogous inequalities in \cite{V} used interval arithmetic, but in this
appendix we take a different approach, based on applying Sturm's theorem to
truncated $q$-series. We have documented the calculations carefully, to
facilitate checking the proof. Computer code for verifying our calculations
is contained in the ancillary file \texttt{appendix.txt}.  The code can be
obtained at \doi{10.4007/annals.2017.185.3.8}, as well as at the arXiv.org
e-print archive, where this paper is available as \arxiv{1603.06518}.  Our
code is for the free computer algebra system PARI/GP (see \cite{PARI}), but
the calculations are simple enough that they are not difficult to check in
any computer algebra system.

To prove each inequality, we approximate the modular form using $q$-series
and prove error bounds for truncating the series, which we then incorporate
by adding them to an appropriate term of the truncated series. The result is
nearly a polynomial in $q$, with the possible exceptions being factors of $t$
(where $z=it$), and we bound those factors so as to reduce to the case of a
polynomial in $q$. Furthermore, we bound any factors of $\pi$ so that the
coefficients become rational. Finally, we use Sturm's theorem with exact
rational arithmetic to verify that the truncated series never changes sign.

To prove the error bounds, we need to control the growth of the coefficients.
We first multiply by $\Delta^2$ to clear the denominators that appear in
\eqref{eq:phidef}, \eqref{eq:psiIdef}, and \eqref{eq:psiSdef}. The advantage
of doing so is that the coefficients of the numerator grow only polynomially.
To estimate the growth rate, we bound the coefficient of $q^n$ in $E_2$ by
$24(n+1)^2$ in absolute value, in $E_4$ by $240(n+1)^4$, and\break in $E_6$
by $504(n+1)^6$.  It is also not difficult to show that the coefficient of
$q^{n/2}$ in $\Theta_{00}^4$, $\Theta_{01}^4$ or $\Theta_{10}^4$ is at most
$24(n+1)^2$ in absolute value.\footnote{Both $\Theta_{00}^4$ and
$\Theta_{10}^4$ have nonnegative coefficients, and their sum is the theta
series of the $D_4$ root lattice in the variable $q^{1/2}$, from which one
can bound their coefficients. Furthermore, $\Theta_{01}^4 = \Theta_{00}^4 -
\Theta_{10}^4$.} Multiplying series is straightforward: if $|a_n| \le
(n+1)^\ell$ and $|b_n| \le (n+1)^m$, then the coefficients of\break
$\big(\sum_n a_n q^n\big)\big(\sum_n b_n q^n\big)$ are bounded by
$(n+1)^{\ell+m+1}$. When we add two $q$-series with coefficients bounded by
different powers of $n+1$, we typically produce an\break upper bound by
rounding up the lower power for simplicity. Using these tech\-niques leads to
explicit polynomial bounds for the coefficients of $\varphi\Delta^2$, $\psi_I
\Delta^2$, and $\psi_S \Delta^2$ by using their definitions in terms of
Eisenstein series and theta functions. These bounds are inefficient, but they
suffice for our purposes.

When $t \ge 1$, $q = e^{-2\pi t}$ is small enough that these coefficient
bounds yield a reasonable error term.  When $t \le 1$, we replace it with
$1/t$ (via $z \mapsto -1/z$) and compute the corresponding $q$-expansion.

\begin{lemma} \label{lemma:phinonpos}
For $t > 0$,
\[
\varphi(it) < 0.
\]
\end{lemma}

\begin{proof}
First, we prove this inequality for $t \ge 1$, in which case $q = e^{-2\pi
t} < 1/535$. The bounds described above show that the coefficient of $q^n$
in $\varphi\Delta^2$ is at most $513200655360(n+1)^{20}$ in absolute value,
and exact computation shows that
\[
\sum_{n = 50}^\infty \frac{513200655360(n+1)^{20}}{535^{n-6}} < 10^{-50}.
\]
Thus, the sum of the absolute values of the terms in $\varphi\Delta^2$ for
$n \ge 50$ amounts to at most $10^{-50} q^6$. Let $\sigma$ be the sum of the
terms with $n<50$. We use Sturm's theorem to check that $\sigma + 10^{-50}
q^6$ never changes sign on $(0,1/535)$ as a polynomial in $q$, and we
observe that it is negative in the limit as $q \to 0$.  This proves that
$\varphi(it) < 0$ for $t \ge 1$.

Using \eqref{eq:10overt}, the bound for $t \le 1$ is equivalent to showing
that
\[
-t^2 \varphi(it) + i t \varphi_1(it) + \varphi_2(it) > 0
\]
for $t \ge 1$. Again we multiply by $\Delta^2$ to control the coefficients.
This case is more complicated, because there are factors of $t$ and $\pi$. We
replace factors of $\pi$ with rational bounds, namely $\lfloor
10^{10}\pi\rfloor/10^{10}$ or $\lceil 10^{10}\pi\rceil/10^{10}$ based on the
sign of the term and whether it is a positive power of $\pi$ (so that we
obtain a lower bound), and we similarly use the bounds $1 \le t \le 1/\big(23
q^{1/2}\big)$; the latter bound follows from $q = e^{-2\pi t}$ and $te^{-\pi
t} \le e^{-\pi} \le 1/23$. To estimate the error bound from truncation, we
use $q^{1/2} < 1/23$; the result is that the error from omitting the $q^n$
terms with $n \ge 50$ is at most $10^{-50} q^6$. These observations reduce
the problem to showing that a polynomial in $q^{1/2}$ with rational
coefficients is positive over the interval $(0,e^{-\pi})$. Using Sturm's
theorem, we check that it holds over the larger interval $(0,1/23)$.
\end{proof}

Note that we could have avoided fractional powers of $q$ in this proof if we
had used a different upper bound for $t$, but fractional powers will be
needed to handle $\psi_S$ and $\psi_I$ in any case.  We will use the bounds
such as $1 \le t \le 1/\big(23 q^{1/2}\big)$ from the preceding proof
systematically in the remaining proofs.

\begin{lemma} \label{lemma:ineq2}
For $t > 0$,
\[
\varphi(it) - \frac{432}{\pi^2} \psi_S(it) > 0.
\]
\end{lemma}

\begin{proof}
We use exactly the same technique as in the proof of
\fullref{Lemma}{lemma:phinonpos}.
For $t \mkern-3.5mu\ge\mkern-3.5mu 1$, removing the $q^{50}$ and higher terms
in the $q$-series for $\big(\varphi - 432 \psi_S/\pi^2\big)\Delta^2$
introduces an error of at most $10^{-50} q^6$, and Sturm's theorem shows that
the resulting polynomial has no sign changes. Note that $\psi_S$ involves
powers of $q^{1/2}$, and so we must view the truncated series as a polynomial
in $q^{1/2}$ rather than $q$.

For $t \le 1$, we apply \eqfullref{relations}{eq:quasimodular} and
\eqref{eq:psiSdef} to reduce the problem to showing that
\[
-t^2 \varphi(it) + i t \varphi_1(it) + \varphi_2(it) -\frac{432}{\pi^2} \psi_I(it) < 0
\]
for $t \ge 1$.  When we multiply by $\Delta^2$ and remove the $q^{50}$ and
higher terms, the error bound is at most $10^{-50} q^6$, and Sturm's theorem
completes the proof.  As in the previous proof, this case involves handling
factors of $t$ and $\pi$, but they present no difficulties.
\end{proof}

Of course these proofs are by no means optimized.  Instead, they were chosen
to be straightforward and easy to describe.

The final inequality we must verify is \eqref{eq:Basymp}:

\begin{lemma} \label{lemma:Bleadingterms}
For all $t \ge 1$,
\[
B(t) > \frac{1}{39}te^{2\pi t}-\frac{10}{117\pi}e^{2\pi t}.
\]
\end{lemma}

\begin{proof}
As usual, we multiply
\[
B(t) -\left( \frac{1}{39}te^{2\pi t}-\frac{10}{117\pi}e^{2\pi t}\right)
\]
by $\Delta^2$ and compute its $q$-series.  Our usual truncation bounds show
that
re\-moving the $q^{50}$ and higher terms introduces an error bound of at most
$10^{-50} q^6$, and Sturm's theorem again completes the proof.
\end{proof}

\end{document}